\documentclass[11pt,english]{scrartcl}
\usepackage[T1]{fontenc}
\usepackage[latin9]{inputenc}
\usepackage[a4paper]{geometry}
\usepackage{babel}
\usepackage{refstyle}
\usepackage{amsthm}
\usepackage{amsmath}
\usepackage{amssymb}
\usepackage{esint}
\usepackage[authoryear]{natbib}
\usepackage{lmodern}
\usepackage{graphics}
\usepackage[kerning=true]{microtype}
\usepackage[unicode=true,pdfusetitle,
 bookmarks=true,bookmarksnumbered=false,bookmarksopen=false,
 breaklinks=false,pdfborder={0 0 0},backref=false,colorlinks=true]
 {hyperref}

\makeatletter

\theoremstyle{plain}
\newtheorem{thm}{Theorem}
\theoremstyle{definition}

\theoremstyle{remark}
\newtheorem{rem}[thm]{Remark}
\theoremstyle{plain}
\newtheorem{lem}[thm]{Lemma}
\theoremstyle{plain}

\setcapindent{0pt}

\global\long\def\e{\mathrm{e}}
\global\long\def\rd{\mathrm{d}}

\global\long\def\bu{\boldsymbol{u}}
\global\long\def\bv{\boldsymbol{v}}

\global\long\def\bx{\boldsymbol{x}}

\global\long\def\be{\boldsymbol{e}}
\global\long\def\bff{\boldsymbol{f}}

\global\long\def\bF{\boldsymbol{F}}

\global\long\def\bn{\boldsymbol{n}}

\global\long\def\bnabla{\boldsymbol{\nabla}}
\global\long\def\bcdot{\boldsymbol{\cdot}}
\global\long\def\bwedge{\boldsymbol{\wedge}}
\global\long\def\bzero{\boldsymbol{0}}
\global\long\def\para{{\raisebox{-4pt}{\scalebox{2.5}{\ensuremath{\cdot}}}}}
\global\long\def\bigO{O\mathopen{}}
\global\long\def\widebar#1{\overline{#1}}

\begin{document}

\title{Optimal asymptotic behavior of the vorticity\\
of a viscous flow past a two-dimensional body}

\author{\href{mailto:julien.guillod@unige.ch}{Julien Guillod} and \href{mailto:peter.wittwer@unige.ch}{Peter Wittwer}\\
{\small{Department of Theoretical Physics,}}\\
{\small{University of Geneva, Switzerland}}}

\date{June 30, 2015}

\maketitle

\begin{abstract}
The asymptotic behavior of the vorticity for the steady incompressible
Navier-Stokes equations in a two-dimensional exterior domain is described
in the case where the velocity at infinity $\bu_{\infty}$ is nonzero.
It is well known that the asymptotic behavior of the velocity field
is given by the fundamental solution of the Oseen system which is
the linearization of the Navier-Stokes equation around $\bu_{\infty}$.
The vorticity has the property of decaying algebraically inside a
parabolic region called the wake and exponentially outside. The previously
proven asymptotic expansions of the vorticity are relevant only inside
the wake because everywhere else the remainder is larger than the
asymptotic term. Here we present an asymptotic expansion that removes
this weakness. Surprisingly, the found asymptotic term is not given
by the Oseen linearization and has a power of decay that depends on
the data. This strange behavior is specific to the two dimensional
problem and is not present in three dimensions.
\end{abstract}
\textit{\small{Keywords:}}{\small{ Navier-Stokes equations, Fluid-structure
interactions, Asymptotic behavior, Vorticity}}\\
\textit{\small{MSC class:}}{\small{ 76D05, 35Q30, 74F10, }}35B40,
35C20, 76M45{\small{, 76D17, 76D25}}{\small \par}

\section{Introduction}

The stationary flow of an incompressible fluid past a body $B$ is
described by the Navier-Stokes equations,\begin{subequations}
\begin{equation}
\begin{aligned}\Delta\bu-\bnabla p & =\bu\bcdot\bnabla\bu+\bff\,, & \bnabla\bcdot\bu & =0\,,\\
\left.\bu\right|_{\partial B} & =\bu^{*}\,, & \lim_{|\bx|\to\infty}\bu & =\bu_{\infty}\,,
\end{aligned}
\label{eq:oseen-ns-0}
\end{equation}
in the domain $\Omega=\mathbb{R}^{2}\setminus\widebar B$, where $\bff$
is the source force, $\bu_{\infty}\neq\bzero$ the velocity at infinity
and $\bu^{*}$ is any boundary condition with no net flux,
\begin{equation}
\int_{\partial B}\bu^{*}\bcdot\bn=0\,.\label{eq:oseen-no-flux}
\end{equation}
\label{eq:oseen-ns-body}\end{subequations}We assume that the body
$B$ is an open bounded domain with smooth boundary. In view of the
symmetries of the equation, we assume without lost of generality that
$\bu_{\infty}=2\be_{1}$ and $\bzero\in B$.

This system has been subject to many investigations, see \citet[Chapter XII]{Galdi-IntroductiontoMathematical2011}
for a complete statement of the main results known for this problem.
\citet{Leray-Etudedediverses1933} has shown the existence of weak
solutions, but with the procedure he used, he was unable to verify
that $\bu$ tends to $\bu_{\infty}$ at large distances. \citet{Gilbarg.Weinberger-AsymptoticPropertiesof1974,Gilbarg.Weinberger-Asymptoticpropertiesof1978}
have shown that any Leray solution $\bu$ either converges at large
distances to some constant vector $\bu_{0}$, or the average of $\bu$
of over circles in the $L^{2}$-norm diverges as the size of the circle
grows. Later on, \citet{Amick-Leraysproblemsteady1988} proved that
if $\bff=\bzero$ and $\bu^{*}=\bzero$, then $\bu\in L^{\infty}$
and therefore $\bu$ converges to a constant $\bu_{0}$ at infinity.
However, the question if $\bu_{0}=\bu_{\infty}$ is still open in
general. In case $\bu_{\infty}\neq\bzero$, \citet{Finn-stationarysolutionsNavier1967,Galdi-ExistenceandUniqueness1993,Galdi-StationaryNavier-Stokesproblem2004}
used the Oseen approximation and a fixed point technique to prove
existence and uniqueness of solutions to \eqref{oseen-ns-body} for
small data. The asymptotic structure of the solutions was presented
by \citet{Babenko-AsymptoticBehaviorof1970} who shows in particular
that velocity behaves at infinity like the Oseen fundamental solution.
The asymptotic expansion of the velocity was also given under more
general assumptions by \citet{Galdi.Sohr-asymptoticstructureof1995,Sazonov-AsymptoticBehaviorof1999}.
The asymptotic behavior of the vorticity was first given by \citet[Theorem 8.1]{Babenko-AsymptoticBehaviorof1970}
only in the wake, and then by \citet[Theorem 3.5']{Clark-VorticityatInfinity1971}.
These two results are relevant only in the wake region\index{Wake!region},
\emph{i.e.} for $\left|\bx\right|-x_{1}\leq1$, because otherwise,
the remainder is larger than the asymptotic term which is given by
the Oseen linearization. In fact, we prove that the true asymptote
which is also valid outside the wake region is not given by the Oseen
linearization.  Under smallness conditions, we show that the asymptote
of the vorticity is given in polar coordinates $(r,\theta)$ by
\[
\omega(\bx)=r^{A\left(1-\cos\theta\right)+B\sin\theta}\left[\frac{\mu(\theta)}{r^{1/2}}+O\left(\frac{1}{r^{1/2+\varepsilon}}\right)\right]\e^{-r\left(1-\cos\theta\right)}\,,
\]
for all $\varepsilon\in(0,1)$, where $A,B\in\mathbb{R}$ depend linearly
on the net force $\bF$ and $\mu$ is a $2\pi$-periodic function
depending on $\bff$ and $\bu^{*}$. Surprisingly the power of decay
of the asymptote depends on the net force $\bF$ and in particular
this contradicts the statement of Theorem~XII.8.4 in \citet{Galdi-IntroductiontoMathematical2011}
for any solution with $\bF\neq\bzero$.

\paragraph{Notation}

For $\bx,\bx_{0}\in\mathbb{R}^{2}$ we use the following notation
\begin{align*}
r & =\left|\bx\right|, & r_{0} & =\left|\bx_{0}\right|, & r_{1} & =\left|\bx-\bx_{0}\right|,\\
\theta & =\angle\bx\,, & \theta_{0} & =\angle\bx_{0}\,, & \theta_{1} & =\angle\left(\bx-\bx_{0}\right),
\end{align*}
where $\angle\bx$ denotes the angle $\theta\in(-\pi,\pi]$ such that
$\bx=\left|\bx\right|\left(\cos\theta,\sin\theta\right)$. For a positive
function $w:\Omega\to\mathbb{R}$, we write $A(\bx,\bx_{0})=O(w(\bx))$,
if for $\bx_{0}$ in a bounded domain, there exists $C>0$ such that
for all $\bx\in\Omega$,
\[
\left|A(\bx,\bx_{0})\right|\leq C\left|w(\bx)\right|.
\]

\section{Asymptote for the linear problem}

It is well-known that the problem \eqref{oseen-ns-body} is related
to the Oseen system which is the linearization of \eqref{oseen-ns-body}
around $\bu=\bu_{\infty}=2\be_{1}$,\index{Oseen equations}
\begin{equation}
\begin{aligned}\Delta\bu-\bnabla p-2\partial_{1}\bu & =\bff\,, & \bnabla\bcdot\bu & =0\,,\\
\left.\bu\right|_{\partial B} & =\bu^{*}-\bu_{\infty}\,, & \lim_{|\bx|\to\infty}\bu & =\bzero\,,
\end{aligned}
\label{eq:oseen-oseen}
\end{equation}
The fundamental solution of the Oseen system is given by\index{Fundamental solution!Oseen equations}
\begin{align*}
\mathbf{E} & =\begin{pmatrix}\partial_{1}\psi-G & \partial_{2}\psi\\
\partial_{2}\psi & -\partial_{1}\psi
\end{pmatrix}, & \be & =-\bnabla H\,,
\end{align*}
with
\begin{align*}
\psi & =\frac{H+G}{2}\,, & H & =\frac{1}{2\pi}\log r\,, & G & =\frac{1}{2\pi}\e^{r\cos\theta}K_{0}(r)\,.
\end{align*}
Denoting by $\mathbf{T}(\bu,p)$ the stress tensor,
\[
\mathbf{T}(\bu,p)=\bnabla\bu+\left(\bnabla\bu\right)^{T}-p\boldsymbol{1}\,,
\]
the Green identity for the Oseen operator is
\begin{multline*}
\int_{\Omega}\left(\bnabla\bcdot\mathbf{T}(\bu,p)-2\partial_{1}\bu\right)\bcdot\bv-\int_{\Omega}\left(\bnabla\bcdot\mathbf{T}(\bv,q)+2\partial_{1}\bv\right)\bcdot\bu\\
=\int_{\partial\Omega}\left(\bv\bcdot\mathbf{T}(\bu,p)-\bu\bcdot\mathbf{T}(\bv,q)-2\bu\bcdot\bv\,\be_{1}\right)\bcdot\bn\,.
\end{multline*}
Therefore, the solution of the Oseen system is given by
\[
\bu(\bx)=\int_{\Omega}\mathbf{E}(\bx-\para)\bff-\int_{\partial\Omega}\left[\mathbf{E}(\bx-\para)\left(\mathbf{T}(\bu,p)-2\bu\otimes\be_{1}\right)+\bu\bcdot\mathbf{T}(\mathbf{E},\boldsymbol{w})(\bx-\para)\right]\bcdot\bn\,,
\]
with
\[
\int_{\partial\Omega}\bu\bcdot\mathbf{T}(\mathbf{E},\boldsymbol{w})(\bx-\para)\bcdot\bn=\int_{\partial\Omega}\bigl[\bn\bcdot\bnabla\mathbf{E}(\bx-\para)\bcdot\bu+\bu\bcdot\bnabla\mathbf{E}(\bx-\para)\bcdot\bn-\boldsymbol{w}(\bx-\para)\,\bu\bcdot\bn\bigr]\,,
\]
where $\para$ denotes a placeholder for the argument over which the
Green function is integrated. In order to obtain the representation
formula for the vorticity, we remark that for any $\boldsymbol{A}\in\mathbb{R}^{2}$,
\[
\bnabla\bwedge\left(\mathbf{E}\bcdot\boldsymbol{A}\right)=\bnabla G\bwedge\boldsymbol{A}\,,
\]
with $G$ as defined above, so we obtain
\begin{align*}
\omega(\bx) & =\int_{\Omega}\bnabla G(\bx-\para)\bwedge\bff-\int_{\partial\Omega}\left[\bnabla G(\bx-\para)\bwedge\left(\mathbf{T}(\bu,p)-2\bu\otimes\be_{1}\right)\right]\bcdot\bn\\
 & \phantom{=}-\int_{\partial\Omega}\left[\bnabla_{\bx}\bigl(\bn\bcdot\bnabla G(\bx-\para)\bigr)\bwedge\bu+\bnabla_{\bx}\left(\bu\bcdot\bnabla G(\bx-\para)\right)\bwedge\bn\right].
\end{align*}
The asymptotic expansions of the fundamental solutions are given
by
\begin{align*}
\mathbf{E}(\bx) & =\frac{-1}{\sqrt{32\pi}}\left(\frac{1}{r^{1/2}}+\bigO\left(\frac{1}{r^{3/2}}\right)\right)\e^{-r\left(1-\cos\theta\right)}\begin{pmatrix}1+\cos\theta & \sin\theta\\
\sin\theta & 1-\cos\theta
\end{pmatrix}+\frac{1}{4\pi r}\begin{pmatrix}\cos\theta & \sin\theta\\
\sin\theta & -\cos\theta
\end{pmatrix},\\
\bnabla G(\bx) & =\frac{1}{\sqrt{8\pi}}\left(1-\cos\theta,-\sin\theta\right)\left(\frac{1}{r^{1/2}}+\bigO\left(\frac{1}{r^{3/2}}\right)\right)\e^{-r\left(1-\cos\theta\right)}\,,
\end{align*}
and for $\left|\alpha\right|=1$,
\begin{align*}
\left|D^{\alpha}\mathbf{E}\right| & \lesssim\left(\frac{\left|\theta\right|}{r^{1/2}}+\frac{1}{r^{3/2}}\right)\e^{-r\left(1-\cos\theta\right)}+\frac{1}{r^{2}}\,,\\
\left|D^{\alpha}\bnabla G\right| & \lesssim\left(\frac{\left|\theta\right|^{2}}{r^{1/2}}+\frac{1}{r^{3/2}}\right)\e^{-r\left(1-\cos\theta\right)}\,.
\end{align*}
If $\bff$ has compact support, the solution of the Oseen equation
\eqref{oseen-oseen} behaves like the fundamental solution,
\begin{align}
\bu(\bx) & =\mathbf{E}(\bx)\bF+\bigO\left(\frac{\left|\theta\right|}{r^{1/2}}+\frac{1}{r^{3/2}}\right)\e^{-r\left(1-\cos\theta\right)}+\bigO\left(\frac{1}{r^{2}}\right),\label{eq:oseen-asy1-lin-u}\\
\omega(\bx) & =\bnabla G(\bx)\bcdot\bF^{\perp}+\bigO\left(\frac{\left|\theta\right|^{2}}{r^{1/2}}+\frac{1}{r^{3/2}}\right)\e^{-r\left(1-\cos\theta\right)}\,,\label{eq:oseen-asy1-lin-w}
\end{align}
where $\bF$ is the net force
\[
\bF=\int_{\Omega}\bff+\int_{\partial B}\left(\mathbf{T}(\bu,p)-2\bu\otimes\be_{1}\right)\bcdot\bn\,.
\]
Explicitly, the asymptotic expansion of the velocity is given by\index{Oseen equations!asymptotic behavior}\index{Asymptotic behavior!in unbounded Lipschitz domains!Oseen solutions}
\begin{equation}
\bu(\bx)=\bu_{w}(\bx)+\bu_{h}(\bx)+\bigO\left(\frac{\left|\theta\right|}{r^{1/2}}+\frac{1}{r^{3/2}}\right)\e^{-r\left(1-\cos\theta\right)}+\bigO\left(\frac{1}{r^{2}}\right),\label{eq:oseen-asy-lin-u}
\end{equation}
where $\bu_{w}$ is the wake part and $\bu_{h}$ a harmonic function,\index{Exact solutions!harmonic}
\begin{align}
\bu_{w}(\bx) & =-\frac{F_{1}\be_{1}}{\sqrt{8\pi}}\frac{1}{r^{1/2}}\e^{-r\left(1-\cos\theta\right)}\,,\label{eq:oseen-def-uw}\\
\bu_{h}(\bx) & =\frac{1}{4\pi r}\begin{pmatrix}\cos\theta & \sin\theta\\
\sin\theta & -\cos\theta
\end{pmatrix}\bF=\frac{F_{1}}{4\pi}\frac{\be_{r}}{r}-\frac{F_{2}}{4\pi}\frac{\be_{\theta}}{r}\,.\label{eq:oseen-def-uh}
\end{align}
In contrast to the asymptotic expansion \eqref{oseen-asy1-lin-u}
or \eqref{oseen-asy-lin-u} of the velocity field, the asymptotic
expansion \eqref{oseen-asy1-lin-w} for the vorticity is only relevant
inside the wake region, because outside the wake region, the remainder
is larger than the asymptotic term. In order to obtain an asymptote
that is relevant in all directions, we have to proceed differently,
by using the following asymptotic expansion,
\[
\bnabla G(\bx-\bx_{0})=\bnabla G(\bx)\e^{r_{0}\left(\cos(\theta-\theta_{0})-\cos\theta_{0}\right)}+\bigO\left(\frac{1}{r^{3/2}}\right)\e^{-r\left(1-\cos\theta\right)}.
\]
By applying this result, we obtain 
\begin{equation}
\omega(\bx)=\bnabla G(\bx)\bcdot\bF^{\perp}(\theta)+\bigO\left(\frac{1}{r^{3/2}}\right)\e^{-r\left(1-\cos\theta\right)}\,,\label{eq:oseen-asy-lin-w}
\end{equation}
where $\bF(\theta)$ is now a function depending on the angle $\theta$,
\[
\bF(\theta)=\int_{\Omega}\e^{r_{0}\left(\cos(\theta-\theta_{0})-\cos\theta_{0}\right)}\bff+\int_{\partial B}\e^{r_{0}\left(\cos(\theta-\theta_{0})-\cos\theta_{0}\right)}\left(\mathbf{T}(\bu,p)-2\bu\otimes\be_{1}\right)\bcdot\bn\,.
\]

To our knowledge, this asymptotic formula is nowhere mentioned in
the literature. With this expression, the asymptote is now detached
in all directions from the remainder.

\section{Asymptote for the nonlinear problem}

For the nonlinear problem \eqref{oseen-ns-body}, neither the asymptotic
behavior \eqref{oseen-asy-lin-u} nor \eqref{oseen-asy-lin-w} is
correct, as shown below. The best results concerning the asymptotic
behavior of $\bu$ and $\omega$ for the nonlinear problem \eqref{oseen-ns-body}
are due to \citet{Babenko-AsymptoticBehaviorof1970}. In particular
he shows the following result:
\begin{thm}[{\citealp[Theorems 6.1 \& 8.1]{Babenko-AsymptoticBehaviorof1970}}]
\label{thm:oseen-babenko}\index{Asymptotic behavior!in unbounded Lipschitz domains!Navier-Stokes solutions!for bu_{infty}neqbzero
@for
$\bu_{\infty}\neq\bzero$}\index{Navier-Stokes equations!in Lipschitz domains!asymptotic behavior!for bu_{infty}neqbzero
@for
$\bu_{\infty}\neq\bzero$}If $\bu$ is a physically reasonable solution of \eqref{oseen-ns-body},
\emph{i.e.} such that $\bu-\bu_{\infty}=O(r^{-1/4-\varepsilon})$
for some $\varepsilon>0$ small enough, then the velocity satisfies
\[
\bu(\bx)-\bu_{\infty}=\mathbf{E}(\bx)\bF+\bigO\left(\frac{\left|\theta\right|\left(\log r,1\right)}{r^{1/2}}+\frac{1}{r}\right)\e^{-r\left(1-\cos\theta\right)}+\bigO\left(\frac{1}{r^{1+\varepsilon}}\right),
\]
and the vorticity satisfies
\[
\omega(\bx)=\bnabla G(\bx)\bcdot\bF^{\perp}+\bigO\left(\frac{\log r}{r^{3/2}}\right)\e^{-\mu r\left(1-\cos\theta\right)}\,,
\]
for some $\mu\in\left(0,1\right)$, where $\bF\in\mathbb{R}^{2}$
is the net force,
\[
\bF=\int_{\Omega}\bff+\int_{\partial B}\left(\mathbf{T}(\bu,p)-\bu\otimes\bu\right)\bcdot\bn\,.
\]

\end{thm}
This result is optimal for the velocity in a sense that the remainder
decays faster than the asymptotic terms. However, for the vorticity,
this result is only relevant in the wake since the remainder is greater
than the asymptotic term if $\theta\neq0$, due to the fact that $\mu\in\left(0,1\right)$.
In fact, we will show that the asymptotic behavior of the vorticity
is not given by the Oseen tensor outside the wake. More precisely
if we consider the vorticity which is given by
\[
\Delta\omega-2\partial_{1}\omega=\boldsymbol{u}\bcdot\bnabla\omega+\bnabla\bwedge\bff\,,
\]
then the Oseen approximation is given by $\bu=\bzero$. We will show
that the linearization that leads to the correct asymptotic behavior
of the nonlinear system is given by linearizing around the harmonic
function $\bu_{h}$ of the Oseen tensor itself \eqref{oseen-def-uh}.
We will show that this new linear system has an asymptotic behavior
where the power of decay itself depends on $\bF$:
\begin{thm}
\label{thm:oseen-main}\index{Asymptotic behavior!in unbounded Lipschitz domains!Navier-Stokes solutions!for bu_{infty}neqbzero
@for
$\bu_{\infty}\neq\bzero$}\index{Navier-Stokes equations!in Lipschitz domains!asymptotic behavior!for bu_{infty}neqbzero
@for
$\bu_{\infty}\neq\bzero$}If there exists $\varepsilon\in\left(0,1/2\right)$ and $A,B\in\mathbb{R}$
such that the solution $\bu$ of \eqref{oseen-ns-body} satisfies
\begin{equation}
\begin{split}\left|\left(\bu-\bu_{\infty}-\bu_{h}\right)\bcdot\be_{r}\right| & \leq\frac{\nu}{r^{1/2}}\e^{-r\left(1-\cos\theta\right)/2}+\frac{\nu}{r^{1+\varepsilon}}\,,\\
\left|\left(\bu-\bu_{\infty}-\bu_{h}\right)\bcdot\be_{\theta}\right| & \leq\frac{\nu}{r}\e^{-r\left(1-\cos\theta\right)/2}+\frac{\nu}{r^{1+\varepsilon}}\,,
\end{split}
\label{eq:oseen-hypothesis}
\end{equation}
where
\[
\bu_{h}=2\bnabla\left(A\log r+B\theta\right)=\frac{2A}{r}\be_{r}+\frac{2B}{r}\be_{\theta}\,,
\]
for some $\nu>0$ small enough, then the solution of \eqref{oseen-ns-body}
with $\bff$ a source term of compact support satisfies for some $C>0$,
\[
\left|\omega\right|\leq C\, r^{A\left(1-\cos\theta\right)+B\sin\theta}r^{-1/2}\e^{-r\left(1-\cos\theta\right)}\,.
\]
Moreover,
\[
\omega(\bx)=r^{A\left(1-\cos\theta\right)+B\sin\theta}\left[\frac{\mu(\theta)}{r^{1/2}}+\bigO\left(\frac{1}{r^{1/2+\varepsilon}}\right)\right]\e^{-r\left(1-\cos\theta\right)}\,,
\]
where $\mu$ is some $2\pi$-periodic function.\end{thm}
\begin{rem}
For $\bu_{\infty}$, $\bff$, and $\bu^{*}$ small enough, the hypothesis
\eqref{oseen-hypothesis} on $\bu$ follows from \thmref{oseen-babenko},
with
\begin{align*}
A & =\frac{F_{1}}{8\pi}\,, & B & =-\frac{F_{2}}{8\pi}\,.
\end{align*}
\end{rem}
\begin{proof}
Let $\bu$ be a solution of \eqref{oseen-ns-body} which by hypothesis
can be written as $\bu=\bu_{\infty}+\bu_{h}+\bar{\bu}$, with $\bar{\bu}$
satisfying
\begin{align*}
\left|\bar{\bu}\bcdot\be_{r}\right| & \leq\frac{\nu}{r^{1/2}}\e^{-r\left(1-\cos\theta\right)/2}+\frac{\nu}{r^{1+\varepsilon}}\,, & \left|\bar{\bu}\bcdot\be_{\theta}\right| & \leq\frac{\nu}{r}\e^{-r\left(1-\cos\theta\right)/2}+\frac{\nu}{r^{1+\varepsilon}}\,.
\end{align*}
This expression is to be understood as To prove the result, we consider
the vorticity equation
\begin{equation}
\Delta\omega-2\partial_{1}\omega-\boldsymbol{u}_{h}\bcdot\bnabla\omega=\bar{\bu}\bcdot\bnabla\omega+\bnabla\bwedge\bff\,.\label{eq:oseen-eq-vort}
\end{equation}
The change of variables,
\[
\omega(r,\theta)=r^{A\left(1-\cos\theta\right)+B\sin\theta}\e^{-r\left(1-\cos\theta\right)}a(r,\theta)\,,
\]
transforms the original equation \eqref{oseen-eq-vort} into
\begin{equation}
\Delta a-2\partial_{r}a-\frac{a}{r}=\boldsymbol{v}\bcdot\bnabla a+\varphi a+R\,,\label{eq:oseen-eq-new}
\end{equation}
where $\boldsymbol{v}$ and $\varphi$ are linearly related to $\boldsymbol{u}_{h}$
and $\bar{\bu}$ and satisfy
\begin{align*}
\left|\boldsymbol{v}\bcdot\be_{r}\right| & \lesssim\frac{\nu}{r^{1/2}}\,, & \left|\boldsymbol{v}\bcdot\be_{\theta}\right| & \lesssim\frac{\nu\log r}{r}\,, & \left|\varphi\right| & \lesssim\frac{\nu}{r^{1+\varepsilon}}\,,
\end{align*}
and where the source term $R$ is given by
\[
R(r,\theta)=r^{-A\left(1-\cos\theta\right)-B\sin\theta}\e^{r\left(1-\cos\theta\right)}\left(\bnabla\bwedge\bff\right)\,.
\]
We will show that the solution $a$ of \eqref{oseen-eq-new} satisfies
the bounds
\begin{align}
\left|a\right| & \lesssim\frac{1}{r^{1/2}}\,, & \left|\be_{r}\bcdot\bnabla a\right| & \lesssim\frac{1}{r^{1+\varepsilon}}\,, & \left|\be_{\theta}\bcdot\bnabla a\right| & \lesssim\frac{1}{r}\,,\label{eq:oseen-bound}
\end{align}
which imply the bound claimed on $\omega$. However in order to prove
\eqref{oseen-bound}, a bootstrap argument in not sufficient. The
idea is to consider the equation \eqref{oseen-eq-new} as a linear
equation in $a$ with $a$ and $\bnabla a$ given on $\partial\Omega$
for fixed $\boldsymbol{v}$, $\varphi$, and $R$ and to construct
a solution $a$ by a fixed point argument. By uniqueness, the solution
constructed by the fixed point argument is equal to the solution whose
existence is assumed by hypothesis.

The fundamental solution of the linear operator defining the left
hand-side of \eqref{oseen-eq-new} is
\begin{equation}
W(\boldsymbol{x},\boldsymbol{x}_{0})=\frac{1}{2\pi}\e^{r-r_{0}}K_{0}\left(\left|\boldsymbol{x}-\boldsymbol{x}_{0}\right|\right),\label{eq:oseen-def-W}
\end{equation}
where $(r_{0},\theta_{0})$ denotes the polar coordinates of $\bx_{0}$.
In view of the Green identity,
\[
\int_{\Omega}\left(\Delta b+2\partial_{r}b+\frac{b}{r}\right)a=\int_{\Omega}\left(\Delta a-2\partial_{r}a-\frac{a}{r}\right)b+\int_{\partial\Omega}\left(a\bnabla b-b\bnabla a+ab\be_{r}\right)\bcdot\bn\,,
\]
the solution can be written as
\begin{equation}
a(\boldsymbol{x})=\int_{\Omega}WR+\int_{\Omega}W\bigl(\varphi a+\boldsymbol{v}\bnabla a\bigr)+\int_{\partial\Omega}\bigl[a\bnabla_{\bx_{0}}W-W\bnabla a+aW\be_{r}\bigr]\bcdot\bn\,,\label{eq:oseen-representation}
\end{equation}
where the integrations are performed over $\bx_{0}$. For $i=1,2,3$,
we denote by $a_{i}(\bx)$ the $i$th term in the expression \eqref{oseen-representation}.

The asymptotic expansion at large $r$, of the fundamental solution
is given by
\begin{align*}
W(\boldsymbol{x},\boldsymbol{x}_{0}) & =\frac{1}{\sqrt{8\pi}}\frac{1}{r^{1/2}}\e^{-r_{0}\left(1-\cos\left(\theta-\theta_{0}\right)\right)}+\bigO\left(\frac{1}{r^{3/2}}\right),\\
\bnabla_{\bx_{0}}W(\boldsymbol{x},\boldsymbol{x}_{0}) & =\frac{\left(\cos\theta-\cos\theta_{0},\sin\theta-\sin\theta_{0}\right)}{\sqrt{8\pi}}\frac{1}{r^{1/2}}\e^{-r_{0}\left(1-\cos\left(\theta-\theta_{0}\right)\right)}+\bigO\left(\frac{1}{r^{3/2}}\right),
\end{align*}
and since $R$ has compact support and $\partial\Omega$ is a bounded,
the first and the third term of \eqref{oseen-representation} have
the claimed asymptotic behavior,
\begin{align*}
a_{1}(\bx) & =\frac{1}{\sqrt{8\pi}}\frac{1}{r^{1/2}}\int_{\mathbb{R}^{2}}\e^{-r_{0}\left(1-\cos\left(\theta-\theta_{0}\right)\right)}R(r_{0},\theta_{0})+\bigO\left(\frac{1}{r^{3/2}}\right)=\frac{\mu_{1}(\theta)}{r^{1/2}}+\bigO\left(\frac{1}{r^{3/2}}\right),\\
a_{3}(\bx) & =\frac{\mu_{3}(\theta)}{r^{1/2}}+\bigO\left(\frac{1}{r^{3/2}}\right),
\end{align*}
where $\mu_{1}$ and $\mu_{3}$ are $2\pi$-periodic functions of
the angle $\theta$,
\begin{eqnarray*}
\mu_{1}(\theta) & = & \frac{1}{\sqrt{8\pi}}\int_{\Omega}r_{0}^{-A\left(1-\cos\theta_{0}\right)-B\sin\theta_{0}}\e^{r_{0}\left(\cos\left(\theta-\theta_{0}\right)-\cos\theta_{0}\right)}\left(\bnabla\bwedge\bff\right)\,,\\
\mu_{3}(\theta) & = & \frac{1}{\sqrt{8\pi}}\int_{\partial\Omega}\e^{-r_{0}\left(1-\cos\left(\theta-\theta_{0}\right)\right)}\left[a\left(\cos\theta-\cos\theta_{0},\sin\theta-\sin\theta_{0}\right)-\bnabla a+a\be_{r}\right]\bcdot\bn\,.
\end{eqnarray*}
In the same way,
\begin{align*}
\bnabla_{\bx}W(\boldsymbol{x},\boldsymbol{x}_{0}) & =\bigO\left(\frac{1}{r^{3/2}}\right), & \bnabla_{\bx}\bnabla_{\bx_{0}}W(\boldsymbol{x},\boldsymbol{x}_{0}) & =\bigO\left(\frac{1}{r^{3/2}}\right),
\end{align*}
and we obtain 
\begin{align*}
\left|\bnabla a_{1}\right| & \lesssim\frac{1}{r^{3/2}}\,, & \left|\bnabla a_{3}\right| & \lesssim\frac{1}{r^{3/2}}\,,
\end{align*}
so $a_{1}$ and $a_{3}$ satisfy \eqref{oseen-bound}. To deal with
the second term $a_{2}$, we first make a fixed point on $a$ in the
space defined by \eqref{oseen-bound} so we have
\[
\left|\varphi a+\boldsymbol{v}\bnabla a\right|\lesssim\frac{\nu}{r^{3/2+\varepsilon}}\,.
\]
By using \lemref{oseen-bounds-W} and \lemref{oseen-bound-exp} we
obtain that the second term $a_{2}$ of \eqref{oseen-representation}
satisfies \eqref{oseen-bound}. Since $\nu$ is small, a fixed point
argument shows the bound \eqref{oseen-bound} claimed on $a$ and
therefore the bound on $\omega$.

In order to prove the asymptotic behavior of $a_{2}$, we show that
\[
I=\int_{\Omega}\biggl|\sqrt{8\pi}W-\frac{1}{r^{1/2}}\e^{-r_{0}\left(1-\cos\left(\theta-\theta_{0}\right)\right)}\biggr|\frac{1}{r_{0}^{3/2+\varepsilon}}\rd\bx_{0}\,,
\]
is bounded by $r^{-1/2-\varepsilon}$. By using the asymptotic expansion
of the Bessel function $K_{0}$, we have $I\leq I_{1}+I_{2}$ where
\begin{align*}
I_{1} & =\int_{\mathbb{R}^{2}}\biggl|\frac{1}{r_{1}^{1/2}}\e^{r-r_{0}-r_{1}}-\frac{1}{r^{1/2}}\e^{-r_{0}\left(1-\cos\left(\theta-\theta_{0}\right)\right)}\biggl|\frac{1}{r_{0}^{3/2+\varepsilon}}\rd\bx_{0}\,,\\
I_{2} & =\int_{\mathbb{R}^{2}}\frac{1}{r_{1}^{3/2}}\e^{r-r_{0}-r_{1}}\frac{1}{r_{0}^{3/2+\varepsilon}}\rd^{2}\bx_{0}\,,
\end{align*}
with $r_{1}=\left|\bx-\bx_{0}\right|$. In view of \lemref{oseen-bound-exp},
$I_{2}$ is bounded by $r^{-3/2}$. For $r_{0}\geq r/2$, the term
$I_{1}$ is bounded by $r^{-1/2+\varepsilon}$. For $r_{0}\leq r/2$,
we have
\[
\biggl|\frac{1}{r_{1}^{1/2}}-\frac{1}{r^{1/2}}\biggl|=\frac{1}{r_{1}^{1/2}r^{1/2}}\frac{r_{0}}{r^{1/2}+r_{1}^{1/2}}\frac{\left|r_{0}-2r\cos\left(\theta-\theta_{0}\right)\right|}{r+r_{1}}\lesssim\frac{r_{0}}{r_{1}^{1/2}r}\,,
\]
\begin{align*}
\left|\e^{r-r_{0}-r_{1}}-\e^{-r_{0}\left(1-\cos\left(\theta-\theta_{0}\right)\right)}\right| & \lesssim\e^{-r_{0}\left(1-\cos\left(\theta-\theta_{0}\right)\right)}\frac{r_{0}^{2}\sin^{2}\left(\theta-\theta_{0}\right)}{r_{1}+r-r_{0}\cos\left(\theta-\theta_{0}\right)}\\
 & \lesssim\e^{-r_{0}\left(1-\cos\left(\theta-\theta_{0}\right)\right)}\frac{r_{0}^{2}\left|\theta-\theta_{0}\right|^{2}}{r}\,.
\end{align*}
Therefore, always for $r_{0}\leq r/2$, we have like in the proof
of \lemref{oseen-bound-exp}, 
\begin{align*}
I_{1} & \lesssim\int_{\mathbb{R}^{2}}\biggl|\frac{1}{r_{1}^{1/2}}-\frac{1}{r^{1/2}}\biggl|\e^{-r_{0}\left(1-\cos\left(\theta-\theta_{0}\right)\right)}\frac{1}{1+r_{0}^{3/2+\varepsilon}}\rd^{2}\bx_{0}\\
 & \phantom{\lesssim}+\int_{\mathbb{R}^{2}}\frac{1}{r_{1}^{1/2}}\biggl|\e^{r-r_{0}-r_{1}}-\e^{-r_{0}\left(1-\cos\left(\theta-\theta_{0}\right)\right)}\biggl|\frac{1}{1+r_{0}^{3/2+\varepsilon}}\rd^{2}\bx_{0}\\
 & \lesssim\frac{1}{r}\int_{\mathbb{R}^{2}}\frac{1}{r_{1}^{1/2}}\e^{-r_{0}\left(1-\cos\left(\theta-\theta_{0}\right)\right)}\frac{1+r_{0}\left|\theta-\theta_{0}\right|^{2}}{1+r_{0}^{1/2+\varepsilon}}\rd^{2}\bx_{0}\lesssim\frac{1}{r^{1/2+\varepsilon}}\,.
\end{align*}
\end{proof}
\begin{lem}
\label{lem:oseen-bounds-W}The Green function \eqref{oseen-def-W}
satisfies
\begin{align*}
\left|W\right| & \lesssim\frac{1}{r_{1}^{1/2}}\e^{r-r_{1}-r_{0}}\,, & \left|\be_{r}\bcdot\bnabla W\right| & \lesssim\frac{1}{r_{1}^{3/2}}\e^{\left(r-r_{1}-r_{0}\right)/2}\,, & \left|\be_{\theta}\bcdot\bnabla W\right| & \lesssim\frac{r_{0}^{1/2}}{r_{1}^{3/2}}\e^{\left(r-r_{1}-r_{0}\right)/2}\,,
\end{align*}
where $r_{1}=\left|\bx-\bx_{0}\right|$.\end{lem}
\begin{proof}
First, the Bessel functions satisfy
\begin{align*}
K_{0}(r) & \lesssim\frac{1}{r^{1/2}}\e^{-r}\,, & K_{1}(r) & \lesssim\frac{1}{r^{1/2}}\e^{-r}\,, & K_{1}(r)-K_{0}(r) & \lesssim\frac{1}{r^{3/2}}\e^{-r}\,,
\end{align*}
so the first bound is proven. Since $r-r_{0}\cos(\theta-\theta_{0})=r_{1}\cos(\theta_{1}-\theta)$,
we have,
\begin{eqnarray*}
\partial_{r}W & = & \frac{1}{2\pi r_{1}}\e^{r-r_{0}}\left[\left(r_{0}\cos(\theta-\theta_{0})-r\right)K_{1}(r_{1})+r_{1}K_{0}(r_{1})\right]\,,\\
 & = & \frac{1}{2\pi}\e^{r-r_{0}}\left[\frac{r_{1}-r+r_{0}\cos(\theta-\theta_{0})}{r_{1}}K_{0}(r_{1})-\cos(\theta-\theta_{1})\left(K_{1}(r_{1})-K_{0}(r_{1})\right)\right]
\end{eqnarray*}
so
\begin{align*}
\left|\partial_{r}W\right| & \lesssim\frac{1}{r_{1}^{3/2}}\e^{r-r_{1}-r_{0}}+\frac{1}{r_{1}}\e^{r-r_{0}}\left(\left|r_{0}+r_{1}-r\right|+r_{0}\left|1-\cos(\theta-\theta_{0})\right|\right)K_{0}(r_{1})\\
 & \lesssim\frac{1}{r_{1}^{3/2}}\e^{r-r_{1}-r_{0}}\left[1+\left|r_{0}+r_{1}-r\right|+r_{0}\left|1-\cos(\theta-\theta_{0})\right|\right]\lesssim\frac{1}{r_{1}^{3/2}}\e^{(r-r_{1}-r_{0})/2}\,.
\end{align*}
For the last bound, we have
\[
\frac{1}{r}\partial_{\theta}W=\frac{1}{2\pi r_{1}}\e^{r-r_{0}}r_{0}\sin(\theta-\theta_{0})K_{1}(r_{1})\,,
\]
so
\[
\left|\frac{1}{r}\partial_{\theta}W\right|\lesssim\frac{1}{r_{1}^{3/2}}\e^{r-r_{1}-r_{0}}r_{0}\left|\sin(\theta-\theta_{0})\right|\lesssim\frac{r_{0}^{1/2}}{r_{1}^{3/2}}\e^{(r-r_{1}-r_{0})/2}\,.
\]
\end{proof}
\begin{lem}
\label{lem:oseen-bound-exp}For $\alpha,\sigma\in\left(0,2\right)$
such that $\alpha+\sigma>3/2$, we have
\[
\int_{\mathbb{R}^{2}}\frac{1}{r_{1}^{\alpha}}\frac{1}{r_{0}^{\sigma}}\e^{r-r_{1}-r_{0}}\rd\bx_{0}\lesssim\frac{1}{r^{\alpha+\sigma-3/2}}+\frac{\left|\log r\right|^{\delta_{\sigma,3/2}}}{r^{\alpha}}+\frac{\left|\log r\right|^{\delta_{\alpha,3/2}}}{r^{\sigma}}\,,
\]
where $r=\left|\bx\right|$, $r_{0}=\left|\bx_{0}\right|$, $r_{1}=\left|\bx-\bx_{0}\right|$
and $\delta_{\alpha,\sigma}$ denotes the Kronecker delta.\end{lem}
\begin{proof}
We have to estimate
\[
I=\int_{0}^{\infty}\int_{-\pi}^{+\pi}\frac{1}{r_{1}^{\alpha}}\frac{1}{r_{0}^{\sigma}}\e^{r-r_{1}-r_{0}}r_{0}\rd\theta_{0}\rd r_{0}\,.
\]
First of all, since
\[
r_{1}^{2}=r^{2}+r_{0}^{2}-2rr_{0}\cos(\theta-\theta_{0})\geq r^{2}+r_{0}^{2}\cos^{2}(\theta-\theta_{0})-2rr_{0}\cos(\theta-\theta_{0})=\left(r-r_{0}\cos(\theta-\theta_{0})\right)^{2}\,,
\]
we have
\[
r_{1}+r_{0}-r\geq r_{0}\left(1-\cos(\theta-\theta_{0})\right)\,,
\]
and therefore
\[
\int_{-\pi}^{+\pi}\e^{r-r_{1}-r_{0}}\rd\theta_{0}\leq\int_{-\pi}^{+\pi}\e^{-r_{0}\left(1-\cos(\theta-\theta_{0})\right)}\rd\theta_{0}\leq\frac{1}{1+r_{0}^{1/2}}\,.
\]
If $r_{0}\leq r/2$, then $r_{1}\geq r/2$ and therefore
\begin{align*}
I & \lesssim\frac{1}{r^{\alpha}}\int_{0}^{r/2}\int_{-\pi}^{+\pi}\frac{1}{r_{0}^{\sigma}}\e^{-r_{0}\left(1-\cos(\theta-\theta_{0})\right)}r_{0}\rd\theta_{0}\rd r_{0}\\
 & \lesssim\frac{1}{r^{\alpha}}\int_{0}^{r/2}\frac{1}{r_{0}^{\sigma-1}}\frac{1}{1+r_{0}^{1/2}}\rd r_{0}\lesssim\frac{\left|\log r\right|^{\delta_{\sigma,3/2}}}{r^{\alpha}}+\frac{1}{r^{\alpha+\sigma-3/2}}\,.
\end{align*}
In the case where $r_{1}\leq r/2$, we have by symmetry the previous
bound with $\alpha$ and $\sigma$ exchanged,
\[
I\lesssim\frac{\left|\log r\right|^{\delta_{\alpha,3/2}}}{r^{\sigma}}+\frac{1}{r^{\alpha+\sigma-3/2}}\,.
\]
Therefore, it remains the case where $r_{0}\geq r/2$ and $r_{1}\geq r/2$,
for which we have
\[
I\lesssim\int_{r/2}^{\infty}\frac{1}{\left(r+\left|r-r_{0}\right|\right)^{\alpha}}\frac{1}{r_{0}^{\sigma-1/2}}\rd r_{0}\lesssim\frac{1}{r^{\alpha+\sigma-3/2}}\,.
\]

\end{proof}

\bibliographystyle{merlin-dot}
\phantomsection\addcontentsline{toc}{section}{\refname}\bibliography{paper}

\end{document}